\documentclass[12pt]{article} 
\usepackage[utf8x]{inputenc}
\usepackage{tikz} %%% parallel arrows 
\usepackage[all]{xy}
\usepackage{authblk}
\usepackage{amsmath,amsthm, amssymb,amsfonts}
\usepackage[hmargin=0.75in,vmargin=1in]{geometry}
\numberwithin{equation}{subsection} %%% ajoute une numérotation des équations avec les sections %%%
\xyoption{arc}
\usepackage{eucal}%%% mathcal change
\usepackage{indentfirst}%%% premier ligne indenter 
\usepackage{mathrsfs}
\usepackage{verbatim}
\usepackage{makeidx}%%%% index 
\usepackage{graphicx}
\usepackage{hyperref}
\usepackage{enumitem} % % % % 
\usepackage{extpfeil} %%%% sert à nommer les fleches avec deux tetes dans un sens : fibrations...
\usepackage{dsfont}
\usepackage[T1]{fontenc}
\newtheorem{thm}{Theorem}[section]
\newtheorem{prop}[thm]{Proposition}
\newtheorem{pdef}[thm]{Proposition-Definition}%%% proposition definition%%%
\newtheorem{lem}[thm]{Lemma}
\newtheorem{sublem}[thm]{Sub-lemma}

\newtheorem{cor}[thm]{Corollary}

\newtheoremstyle{bidule}% name
{3pt}% Space above
{3pt}% Space below
{}% Body font
{}% Indent amount
{\scshape}% Theorem head font
{.}% Punctuation after theorem head
{.5em}% Space after theorem head
{}% Theorem head spec (can be left empty, meaning `normal')
\newtheorem{df}[thm]{Definition}
\theoremstyle{definition}
\newtheorem{rmk}{Remark}[section]%%%

\newtheorem*{warn}{Warning}

\newtheorem{nota}{Notation}[section]

\newcommand{\C}{\mathcal{C}}
\newcommand{\Ub}{\mathcal{U}}%%% Ub comme Oublie%%% forgetful functor%%%%%
\newcommand{\F}{\mathcal{F}}

\newcommand{\D}{\mathcal{D}}
\newcommand{\Ba}{\mathcal{B}}
\newcommand{\Ga}{\mathcal{G}}

\newcommand{\Aa}{\mathcal{A}}

\newcommand{\M}{\mathscr{M}}

\newcommand{\J}{\mathcal{J}} %% index category for limit and colimits%%%

\newcommand{\Ea}{\mathcal{E}} %%%

\newcommand{\G}{\mathcal{G}}

\renewcommand{\to}{\longrightarrow}

\newcommand{\ul}{\underline}

\newcommand{\Ob}{\text{Ob}}% set of objects  
\newcommand{\uc}{\mathds{1}} % the unit category

\newcommand{\tx}{\text}

\renewcommand{\to}{\longrightarrow}
\DeclareMathOperator\Id{Id}
\DeclareMathOperator\Hom{Hom}
\DeclareMathOperator\HOM{\ul{Hom}}

\DeclareMathOperator\Set{\textbf{Set}} % Category of Set
\DeclareMathOperator\Lax{Lax}%% Lax functors 
\DeclareMathOperator\dCat{\mathbf{2-Cat}} %%
\DeclareMathOperator\mdcat{\M_{lax}(\dCat)}%% Lax functors 
\DeclareMathOperator\dfunc{2\tx{-Func}}%% Lax functors 
\DeclareMathOperator\El{\tx{El}}%% element Lax functors 

\title{A calculus on lax functors}
\author{Hugo V. Bacard }
\affil{Western University}
\date{}

\begin{document}
\maketitle
\begin{abstract}
We verify that Kelly's constructions of the internal Hom for enriched categories extends naturally to lax functors taking their values in a symmetric monoidal category. Our motivation is to set up a `calculus on lax functors' that will host the theory of weakly enriched categories that are defined by lax diagrams. 
\end{abstract}
\setcounter{tocdepth}{1}
\tableofcontents 
\section{Introduction}
The theory of enriched categories can be viewed, in a certain way, as a small part of the theory of lax functors. This remark goes back to Bénabou \cite{Ben2} and Street \cite{Str} who  observed that enriched categories are the same as thing as lax functors indexed by \emph{indiscrete categories} (also called \emph{
chaotic}, or \emph{coarse} categories). But the theory of enriched categories has not been fully developed within the theory of lax functors, as far as the author knows.  And the purpose of the paper is to take a little step toward this. A \emph{calculus for lax diagrams} aims to provide the necessary tools and constructions that are needed to develop a theory for weakly enriched categories as initiated in \cite{COSEC1}. It should be an analogue of the theory developed by Joyal \cite{Joyal_qcat} and Lurie \cite{Lurie_HTT} on simplicial sets for  quasicategories.\\

We show that there exists an internal Hom for lax diagrams indexed by strict $2$-categories and taking their values in a symmetric monoidal category $\M=(\ul{M}, \otimes,I)$. Here we assume $\M=(\ul{M}, \otimes,I)$ to have nice properties e.g, cocomplete and complete.\\
To do this, we consider the analogue of a natural transformation in the context of lax functors; we call them \emph{nerve-transformations} (Definition \ref{def-nerve-trans}). With this ingredient at hand \emph{`Kelly calculus'} applies naturally except that we don't mention the notion of \emph{extraordinary naturality}. With some basic category theory we get  our main result (Theorem \ref{thm-adjunction}) which gives the desired adjunction.\\

We shall close this introduction with few comments. 
\begin{enumerate}
\item As for enriched categories, the existence of an internal Hom for lax functors, allows to construct new symmetric \ul{closed} monoidal categories from old ones. In particular when $\M$ is the category of chain complexes over some ring, then iteratively one builds higher linear categories (using for example the co-Segal formalism). The author was told that an idea of \emph{deformation of higher linear categories} becomes necessary in many fields. 
\item Given a lax functor $\F$, we can ask what should be the corresponding constructions that exist for enriched categories e.g, free completion of $\F$, Cauchy completion (Karoubi envelope), etc. 
\item In the presence of weak equivalences in $\M$, we're mainly interested in lax diagrams $\F$, that are \emph{category-like}. There are triples $(\F,\C,\Aa)$ where $\F: \C \to \M$ is a lax functor; $\Aa \subset \C$ is a sub-$2$-category such that the restriction $\F_{|\Aa}: \Aa \to \M$ sends every $2$-morphism of $\Aa$ to a weak equivalence in $\M$. And the theory of lax functors should be developed to study their homotopy theory. Among such lax functor, we have the one we called \emph{co-Segal enriched categories} (see \cite{COSEC1}) which originally motivated all of this. 
\item Finally there is no reason to limit this discussion for monoidal categories. One should be able to adapt the constructions for a (symmetric) monoidal $2$-category. We didn't go through that here. 
\end{enumerate}

\begin{warn}
In this paper all size issues on 2-categories and categories have been left aside. Some notations may look ``bizarre'' but we use them to keep track of everything. For example when two pairs of morphisms $(f,g)$ and $(r,s)$ have the same composite i.e, $f\otimes g= r \otimes s= u$ and given a functor $\F$ we will sometime write 
$$\F(u)=\frac{\F(f\otimes g)}{\F(r \otimes s)}.$$
In some diagrams we sometimes replace `$\otimes$' by a `$\cdot$' to get a smaller diagram.
\end{warn}

\section{Tensor product of lax functors}
Let $\C$, $\D$ be two strict $2$-categories and $\M=(\ul{M},\otimes,I)$ be a symmetric monoidal \ul{closed} category. As usual we will identify $\M$ with a $2$-category with a single object. In the following we are interested in the Lax-Yoneda functor of points $\Lax(-,\M)$.  For each $\C$ we have a category $\Lax(\C,\M)$ of lax functors and transformations that are icons in the sense of Lack \cite{Lack_icons}. If $\F \in  \Lax(\C,\M)$ and $\Ga \in \Lax(\D,\M)$, a morphism  $\sigma: \F \to \Ga$ is by definition a pair $(p,\sigma)$ where, $p: \C \to \D$ is a strict $2$-functor and $\sigma:  \F \to p^{\star}\Ga$ is a morphism in $\Lax(\C,\M)$. \ \\ We will denote by $\M_{lax}(\dCat)$ the fibred category over $\dCat$ where the fiber of $\C$ is $\Lax(\C,\M)$.\ \\

We have an obvious point-wise tensor product functor:
$$\boxtimes: \Lax(\C,\M) \times \Lax(\D,\M) \to \Lax(\C \times \D,\M)$$  
that makes $\M_{lax}(\dCat)$  a symmetric monoidal category whose unit object is the lax morphism $\J: \uc \to \M$ which expresses $I$ as the trivial monoid.
\paragraph{Our main goal} We want to find out if there is an internal $\Hom$ in $\mdcat$ with respect to $\boxtimes$ when $\M$ has itself an internal $\HOM$. This will generalized the functor category between $\M$-categories (see \cite{Ke}).

\section{Generalization of natural transformations}
Let $\F \in \Lax(\C,\M)$ and $\G \in \Lax(\D,\M)$ be two lax morphisms and $\sigma_1=(p_1,\sigma_1)$, $  \sigma_2=(p_2,\sigma_2)$ be two parallel morphisms in $\mdcat$ from  $\F$ to $\Ga$.  The definition we consider below makes sense even if we consider a $2$-category $\M$.

\begin{df}\label{def-nerve-trans}
A \textbf{nerve-transformation} $\eta: \sigma_1 \to \sigma_2$ consists of the following data and axioms.\ \\ 
\ \\
\textbf{\ul{Data:}} 
\begin{itemize}[label=$-$]
\item a \ul{strict}  $2$-natural-transformation of $2$-functors $\alpha: p_1 \to p_2$;
\item  A family $\{\eta_c: \Id_{\ast} \to \G( \alpha_c)\}_{c \in \Ob(\C)}$ of  $2$-morphisms of $\M$; where $\alpha_c: p_1(c) \to p_2(c)$ is the component of $\alpha$. 
\end{itemize}
\ \\
\textbf{\ul{Axiom:}} 
For every $1$-morphism $c \xrightarrow{f}  d$ of $\C$, the following \emph{hexagon} commutes:
\[
\xy
(-30,10)*+{\F(f)}="A";
(0,18)*+{I \otimes \F (f)}="W";
(0,0)*+{\F(f) \otimes I}="X";
(30,0)*+{\G p_2(f) \otimes  \G \alpha_c}="Y";
(30,18)*+{\G \alpha_d \otimes \G p_1(f)}="E";
(70,10)*+{\ \ \ \ \ \ \ \ \ \  \ \ \ \ \ \G[p_2(f) \otimes \alpha_c]= \G[\alpha_d \otimes p_1(f)]}="B";
{\ar@{->}^-{\sigma_2 \otimes \eta_c}"X";"Y"};
{\ar@{->}^-{\eta_d \otimes \sigma_1}"W";"E"};
{\ar@{->}^-{\cong}"A";"W"};
{\ar@{->}^-{\cong}"A";"X"};
{\ar@{->}^-{}"Y";"B"};
{\ar@{->}^-{}"E";"B"};
\endxy
\]
\end{df}

\begin{nota}
We will denote by $\Hom(\sigma_1, \sigma_2)$ the `set' of all nerve-transformations $\eta: \sigma_1 \to \sigma_2$. 
\end{nota}

\subsection{Composition of nerve-transformations}

Given $\eta=(p_1 \xrightarrow{\alpha} p_2,\{\eta_c\}) \in  \Hom(\sigma_1, \sigma_2)$ and $\eta'=( p_2 \xrightarrow{\alpha'} p_3,\{\eta'_c\}) \in \Hom(\sigma_2,  \sigma_3)$, we define the \textbf{horizontal composite} nerve-transformation $\eta' \otimes \eta: \sigma_1 \to \sigma_3$ given by:
\begin{itemize}[label=$-$]
\item the strict transformation $\alpha' \otimes \alpha: p_1 \to p_3$
\item the family of $2$-morphisms 
$$(\eta' \otimes \eta)_c: I \xrightarrow{\cong} I \otimes I \xrightarrow{\eta'_c \otimes \eta_c} \G(\alpha'_c) \otimes \G(\alpha_c) \xrightarrow{\tx{laxity}} \G[(\alpha' \otimes \alpha)_c].$$
\end{itemize}
We shall now check that these data satisfy the coherence conditions. This is a consequence of the fact that each of the following \emph{hexagons} commutes:

\[
\xy
(-20,10)*+{\F(f)}="A";
(0,18)*+{I \otimes \F (f)}="W";
(0,0)*+{\F(f) \otimes I}="X";
(30,0)*+{\G p_2(f) \otimes  \G \alpha_c}="Y";
(30,18)*+{\G \alpha_d \otimes \G p_1(f)}="E";
(60,10)*+{ \frac{\G[\alpha_d \otimes p_1(f)]}{\G[p_2(f) \otimes \alpha_c]}} ="B";
{\ar@{->}^-{\sigma_2 \otimes \eta_c}"X";"Y"};
{\ar@{->}^-{\eta_d \otimes \sigma_1}"W";"E"};
{\ar@{->}^-{\cong}"A";"W"};
{\ar@{->}^-{\cong}"A";"X"};
{\ar@{->}^-{}"Y";"B"};
{\ar@{->}^-{}"E";"B"};
\endxy
\]
\[
\xy
(-20,10)*+{\F(f)}="A";
(0,18)*+{I \otimes \F (f)}="W";
(0,0)*+{\F(f) \otimes I}="X";
(30,0)*+{\G p_3(f) \otimes  \G \alpha'_c}="Y";
(30,18)*+{\G \alpha'_d \otimes \G p_2(f)}="E";
(60,10)*+{ \frac{\G[\alpha'_d \otimes p_2(f)]}{\G[p_3(f) \otimes \alpha'_c]}} ="B";
{\ar@{->}^-{\sigma_3 \otimes \eta'_c}"X";"Y"};
{\ar@{->}^-{\eta'_d \otimes \sigma_2}"W";"E"};
{\ar@{->}^-{\cong}"A";"W"};
{\ar@{->}^-{\cong}"A";"X"};
{\ar@{->}^-{}"Y";"B"};
{\ar@{->}^-{}"E";"B"};
\endxy
\]

The idea is that the first hexagon can be considered as `$\eta_d \otimes \sigma_1 = \sigma_2 \otimes \eta_c$' whereas the other one is `$\eta'_d \otimes \sigma_2 = \sigma_3 \otimes \eta'_c$'. And if we tensor on both sides by $\eta'_d$ in the first equality, one successively establishes the following.
\begin{equation*}
\begin{split}
\eta'_d \otimes (\eta_d \otimes \sigma_1)&=\eta'_d \otimes (\sigma_2 \otimes \eta_c)  \\
&= (\eta'_d \otimes \sigma_2) \otimes \eta_c  \\
&= (\sigma_3 \otimes \eta'_c ) \otimes \eta_c \\
&= \sigma_3 \otimes (\eta'_c  \otimes \eta_c). \\
\end{split}
\end{equation*}
So that we have `($\eta'_d \otimes \eta_d) \otimes \sigma_1 = \sigma_3 \otimes (\eta'_c  \otimes \eta_c)$' which means that $\eta' \otimes \eta$ satisfies the coherence condition.  This is the key idea behind and we give the different steps diagrammatically below. 

\paragraph*{Step 1} Tensor (on the left)  the first hexagon by $\eta'_d: I \to \G(\alpha'_d)$ to get a new one:
\[
\xy
(-20,10)*+{I \otimes \F(f)}="A";
(0,18)*+{I \otimes I \otimes \F (f)}="W";
(0,0)*+{I \otimes \F(f) \otimes I}="X";
(40,0)*+{\G \alpha'_d \otimes \G p_2(f) \otimes  \G \alpha_c}="Y";
(40,18)*+{\eta'_d \otimes \G \alpha_d \otimes \G p_1(f)}="E";
(70,10)*+{ \G \alpha'_d \otimes \frac{\G[\alpha_d \otimes p_1(f)]}{\G[p_2(f) \otimes \alpha_c]}} ="B";
(120,10)*+{ \frac{\G[\alpha'_d \otimes \alpha_d \otimes p_1(f) ]}{\G[\alpha'_d \otimes p_2(f) \otimes \alpha_c]}}="C";
{\ar@{->}^-{\eta'_d \otimes \sigma_2 \otimes \eta_c}"X";"Y"};
{\ar@{->}^-{\eta'_d \otimes \eta_d \otimes \sigma_1}"W";"E"};
{\ar@{->}^-{\cong}"A";"W"};
{\ar@{->}^-{\cong}"A";"X"};
{\ar@{->}^-{}"Y";"B"};
{\ar@{->}^-{}"E";"B"};
%%% last laxity%%%%
{\ar@{-->}^-{}"B";"C"};
\endxy
\]

\paragraph*{Step 2} Tensor (on the right) the second hexagon by $\eta_c$ to get the new one:
\[
\xy
(-20,10)*+{\F(f) \otimes I}="A";
(0,18)*+{I \otimes \F (f) \otimes I}="W";
(0,0)*+{\F(f) \otimes I \otimes I}="X";
(40,0)*+{\G p_3(f) \otimes  \G \alpha'_c \otimes \G \alpha_c}="Y";
(40,18)*+{\G \alpha'_d \otimes \G p_2(f)\otimes \G \alpha_c}="E";
(70,10)*+{ \frac{\G[\alpha'_d \otimes p_2(f)]}{\G[p_3(f) \otimes \alpha'_c]} \otimes \G \alpha_c} ="B";
(120,10)*+{ \frac{\G[\alpha'_d \otimes p_2(f) \otimes \alpha_c ]}{\G[p_3(f) \otimes \alpha'_c \otimes \alpha_c]}}="C";
{\ar@{->}^-{\sigma_3 \otimes \eta'_c \otimes \eta_c}"X";"Y"};
{\ar@{->}^-{\eta'_d \otimes \sigma_2 \otimes \eta_c }"W";"E"};
{\ar@{->}^-{\cong}"A";"W"};
{\ar@{->}^-{\cong}"A";"X"};
{\ar@{->}^-{}"Y";"B"};
{\ar@{->}^-{}"E";"B"};
%%% last laxity%%%%
{\ar@{-->}^-{}"B";"C"};
\endxy
\]

\paragraph*{Step 3} Join the last laxity map and `glue' the two hexagons using the symmetry of $\otimes$. This is possible thanks to the coherence for the lax functor $\G$, which says that given $3$-composable $1$-morphisms $s,t,u$ in $\D$, then the diagram below commutes:
\[
\xy
(-30,10)*+{\G s \otimes \G t  \otimes \G u}="A";
(0,18)*+{\G (s \otimes  t)  \otimes \G u}="W";
(0,0)*+{\G s \otimes \G (t  \otimes  u)}="X";
(30,10)*+{\G(s \otimes t \otimes u)} ="B";
{\ar@{->}^-{}"A";"W"};
{\ar@{->}^-{}"A";"X"};
{\ar@{->}^-{}"X";"B"};
{\ar@{->}^-{}"W";"B"};
\endxy
\]

So we have a composite $\eta' \otimes \eta: \sigma_1 \to \sigma_3$ as claimed. This define a function 
$$\otimes: \Hom(\sigma_1, \sigma_2) \times  \Hom(\sigma_2, \sigma_3) \to \Hom(\sigma_1, \sigma_3).$$

Again thanks to the coherence for the lax functor $\G$, it's easy to see that this composition is associative so that we have a category $nT(\F,\G)$ whose objects are lax transformations $\sigma: \F \to \G$ and whose morphisms are the nerve-transformations. \\

Given $\sigma=(p, \sigma): \F \to \G$, the identity nerve transformation $\Id_{\sigma}$ is the one given by $\Id_p$ and the identity $\Id_{I}: I \to I$.
\section{The $2$-category of elements of a lax functor}
Let $\F: \C \to \M$ be a lax functor. The purpose of this section is to construct, functorially, a $2$-category $\El(\F)$ equipped with a canonical $2$-functor $\Phi:\El(\F) \to \C$ which is locally a Grothendieck op-fibration.
\paragraph{Description of  $\El(\F)$} The objects of $\El(\F)$ are the same as of $\C$. If $A,B \in \C$ then the category of morphisms $\El(\F)(A,B)$ is the category of element of the $\Set$-valued functor:
$$ \C(A,B) \xrightarrow{\F_{AB}} \ul{M} \xrightarrow{\Hom(I,-)} \Set.$$

We have an op-fibration $ \Phi: \El(\F)(A,B) \to \C(A,B)$. In elementary words, a $1$-morphism in 
$\El(\F)(A,B)$ is an `\emph{element}' $I \to \F(k)$ where $k: A \to B$ is a $1$-morphism in $\C(A,B)$. A $2$-morphism between two such elements is  a $2$-morphism $\alpha : k \to k'$ in $\C$, such that the obvious triangle commutes.
The $2$-functor $\Phi$ is the identity on objects and takes $I \to \F(k)$ to $k$; and takes morphism of elements to $\alpha$. \ \\

The composition is the natural one i.e, if $I \to \F(k), I \to \F(k')$ are two composable elements, then the composite element is given by:
$$ I \cong I \otimes I  \to F(k) \otimes F(k') \xrightarrow{\tx{laxity}} F(k\otimes k').$$

And we leave the reader to check that:
\begin{prop}
The above data define a $2$-category $\El(\F)$ together with a $2$-functor 
$$\Phi:\El(\F) \to \C.$$

The operation $\F \mapsto \El(\F)$  defines a $2$-functor $\El: \M_{lax}(\dCat) \to \dCat$.
\end{prop}

%%%%%%%%%%%%%%%%%%%%%%%%%%%%%%%%%%%%%%%%%%%%%%%%
%%%%%%%%%%%%%%%%%%%%%%%%%%%%%%%%%%%%%%%%%%%%%%%%
%%%%%%%%%%%%%%%%%%%%%%%%%%%%%%%%%%%%%%%%%%%%%%%%
%%%%%%%%%%%%%%%%%%%%%%%%%%%%%%%%%%%%%%%%%%%%%%%%
%%%%%%%%%%%%%%%%%%%%%%%%%%%%%%%%%%%%%%%%%%%%%%%%
%%%%%%%%%%%%%%%%%%%%%%%%%%%%%%%%%%%%%%%%%%%%%%%%

\subsection{The $2$-category $[\F,\G]$}
We define previously a $1$-category $nT(\F,\G)$ of lax-transformations and nerve-transformations. But there is a natural notion of $2$-morphisms in $nT(\F,\G)$.
\begin{df}
Let $\eta=(\alpha,\eta)$ and $\eta'=(\alpha', \eta')$ be two parallel morphisms from $\sigma_1 \to \sigma_2$.\ \\
A morphism $\Gamma: \eta \to \eta'$ is a \textbf{modification} $\Gamma: \alpha \to \alpha'$ such that for every object $c \in \C$ the following diagram of $2$-morphisms commute:
 
 \[
\xy
(-30,10)*+{\Id_{\ast}}="A";
(0,18)*+{\G \alpha_c}="W";
(0,0)*+{\G \alpha'_c}="X";
{\ar@{->}^-{\eta_c}"A";"W"};
{\ar@{->}^-{\eta'_c}"A";"X"};
{\ar@{->}^-{\G \Gamma_c}"W";"X"};
\endxy
\]
\end{df}

\begin{prop}
The category $nT(\F,\G)$ equipped with the morphism of nerve transformation is a \ul{strict} $2$-category. We will write $[\F,\G]$ this $2$-category.
\end{prop}

\begin{proof}
Elementary. Left to the reader.
\end{proof}
\section{The classifying $2$-category}
Recall that there is a $2$-category $\dfunc(\C, \D)$ consisting of $2$-functors, strict transformations and modifications. If $\F: \C \to \M$ and $\G: \D \to \M$ are as above, there is a canonical $2$-functor
\begin{align*}
\Ub: [\F, \G] \to \dfunc(\C,\D) \\
\intertext{given by}
\sigma=(p,\sigma) \mapsto p \\
\eta=(\alpha, \eta: \sigma \to \sigma')  \mapsto \alpha \\
(\eta \xrightarrow{\Gamma} \eta') \mapsto  (\alpha \xrightarrow{|\Gamma|} \alpha')
\end{align*}

With this functor, we will say that a nerve transformation $\eta: \sigma \to \sigma'$ \textbf{is over} $\alpha$ if $\Ub \eta =\alpha$.\ \\
The next ingredients that we need  are the \textbf{classifying objects} for nerve transformations over $\alpha$.

\begin{df}
Let $\sigma_1=(p_1,\sigma_1)$, $\sigma_2=(p_2, \sigma_2)$ be two parallel transformations from $\F$ to $\G$, and $\alpha: p_1 \to p_2$ be a transformation in $\dfunc(\C,\D)$. \ \\

Define the \textbf{classifying object} for nerve transformations from $\sigma_1$ to $\sigma_2$ over $\alpha$, to be the equalizer 
$$\int_{\sigma_1}^{\sigma_2} \G \alpha  \to  \prod_{c \in \Ob(\C)} \G \alpha_c \rightrightarrows \prod_{c \xrightarrow{f} d} \frac{\HOM_{\M}(\F(f), \G[\alpha_d \otimes p_1(f)])}{ \HOM_{\M}(\F (f), \G[p_2(f) \otimes \alpha_c])}$$
where the parallel arrows are  the adjoint transpose (from the adjunction of $\HOM$ and $\otimes$ in $\M$) of: 
$$ \G \alpha_d \otimes \F(f) \xrightarrow{\Id \otimes \sigma_1} \G \alpha_d \otimes \G p_1(f) \to  \G[\alpha_d \otimes p_1(f)] $$
$$\F(f) \otimes \G \alpha_c \xrightarrow{ \sigma_2 \otimes \Id} \G p_2(f) \otimes \G \alpha_c \to \G[p_2(f) \otimes \alpha_c].  $$  
\end{df} 

\begin{rmk}\label{rmk_eval}
It's important to observe that by definition for any $f: c \to d$ and $\alpha: p_1 \to p_2$ as above, there is a \textbf{canonical evaluation} map:
$$(\int_{\sigma_1}^{\sigma_2} \G \alpha) \otimes \F(f) \to  \frac{\G[\alpha_d \otimes p_1(f)] }{\G[p_2(f) \otimes \alpha_c]}.$$

That map is given by composing, respectively 
$$(\int_{\sigma_1}^{\sigma_2} \G \alpha) \otimes \F(f) \xrightarrow{pr_d \otimes \Id} \G \alpha_d \otimes \F(f)$$ 
and 
$$(\int_{\sigma_1}^{\sigma_2} \G \alpha) \otimes \F(f) \xrightarrow[\cong]{\tx{symmetry}} \F(f) \otimes (\int_{\sigma_1}^{\sigma_2} \G \alpha) \xrightarrow{\Id \otimes pr_c}\F(f) \otimes \G \alpha_c$$
with the adjoint transpose map in the definition.
\end{rmk}

\begin{prop}\label{prop-funct-integ}
 If $\Gamma: \alpha \to \alpha'$ is a modification then:
\begin{enumerate}
\item there is a canonical map in $\M$:
$$\int_{\sigma_1}^{\sigma_2} \G \Gamma : \int_{\sigma_1}^{\sigma_2} \G \alpha  \to \int_{\sigma_1}^{\sigma_2} \G \alpha' .$$
\item the morphism $ \int_{\sigma_1}^{\sigma_2} \G \Gamma$ distribute over the vertical composition of modifications i.e, the assignment 
$$ \alpha \mapsto \int_{\sigma_1}^{\sigma_2} \G \alpha $$
is functorial in $\alpha$.  

\end{enumerate}
 \end{prop}

\begin{proof}
Assertion $(2)$ is left to the reader. We will only give a proof of Assertion $(1)$.  \\

Our desired map will follow by universal property of equalizers, if we show that  for any $f: c \to d$, the following paths give the same map:
$$\int_{\sigma_1}^{\sigma_2} \G \alpha \otimes \F(f) \xrightarrow{pr_d \otimes \Id} \G \alpha_d \otimes \F(f) \xrightarrow{\G \Gamma_d \otimes \Id} \G \alpha'_d \otimes \F(f) \to \frac{\G[\alpha'_d \otimes p_1(f)] }{\G[p_2(f) \otimes \alpha'_c]}$$

$$\int_{\sigma_1}^{\sigma_2} \G \alpha \otimes \F(f) \xrightarrow[\cong]{\tx{sym}} \F(f) \otimes \int_{\sigma_1}^{\sigma_2} \G \alpha \xrightarrow{\Id \otimes pr_c } \F(f) \otimes \G \alpha_c  \xrightarrow{ \Id \otimes \G \Gamma_c } \F(f) \otimes \G \alpha'_c   \to \frac{\G[\alpha'_d \otimes p_1(f)] }{\G[p_2(f) \otimes \alpha'_c]}.$$

Since $\Gamma: \alpha \to \alpha'$ is a modification, one has by definition an equality of $2$-morphisms
$$ p_2(f) \otimes \Gamma_c= \Gamma_d \otimes p_1(f)$$
and therefore after applying  $\G$, we get an equality
$$  \G[p_2(f) \otimes \Gamma_c]=  \G[\Gamma_d \otimes p_1(f)].$$

Now thanks to the functoriality of the laxity maps of $\G$ we get two commutative diagrams:
\[
\xy
(-10,18)*+{\G \alpha_d \otimes \F(f)}="W";
(-10,0)*+{\G \alpha'_d \otimes \F(f)}="X";
(20,0)*+{\G \alpha'_d \otimes \G p_1(f)}="Y";
(20,18)*+{\G \alpha_d \otimes \G p_1(f)}="E";
(50,18)*+{\G [\alpha_d \otimes p_1(f)]}="B";
(50,0)*+{\G [\alpha'_d \otimes  p_1(f)]}="C";
{\ar@{->}^-{}"X";"Y"};
{\ar@{->}_-{\G(\Gamma_d) \otimes \Id}"W";"X"};
{\ar@{->}^-{}"W";"E"};
{\ar@{->}^-{}"E";"B"};
{\ar@{->}^-{}"Y";"C"};
{\ar@{->}^-{\G[\Gamma_d \otimes p_1(f)]}"B";"C"};
\endxy
\]
\[
\xy
(-10,18)*+{ \F(f) \otimes \G \alpha_c }="W";
(-10,0)*+{\F(f) \otimes \G \alpha'_c}="X";
(20,0)*+{\G p_2(f) \otimes \G \alpha'_c}="Y";
(20,18)*+{\G p_2(f) \otimes \G \alpha_c}="E";
(50,18)*+{\G [p_2(f)\otimes \alpha_c ]}="B";
(50,0)*+{\G [p_2(f)\otimes \alpha'_c]}="C";
{\ar@{->}^-{}"X";"Y"};
{\ar@{->}_-{\Id \otimes \G(\Gamma_c) }"W";"X"};
{\ar@{->}^-{}"W";"E"};
{\ar@{->}^-{}"E";"B"};
{\ar@{->}^-{}"Y";"C"};
{\ar@{->}^-{\G[ p_2(f) \otimes \Gamma_c]}"B";"C"};
\endxy
\]
\ \\

If we precompose these diagrams respectively by 
$$\int_{\sigma_1}^{\sigma_2} \G \alpha \otimes \F(f) \xrightarrow{pr_d \otimes \Id} \G \alpha_d \otimes \F(f)$$ and  $$\int_{\sigma_1}^{\sigma_2} \G \alpha \otimes \F(f) \xrightarrow[\cong]{\tx{sym}} \F(f) \otimes \int_{\sigma_1}^{\sigma_2} \G \alpha  \xrightarrow{\Id \otimes pr_c } \F(f) \otimes \G \alpha_c $$
 we get two new commutative squares sharing the \ul{same} bottom horizontal map; this is because by definition, $\int_{\sigma_1}^{\sigma_2} \G \alpha$ is some equalizer.\ \\ 
 
 As these squares commute it's easy to see that the two maps we wanted to compare are in fact the same and the assertion follows. 
\end{proof}

\begin{pdef}
The following data define a strict $2$-category $B(\F, \G)$
\begin{itemize}[label=$-$]
\item the objects are all transformations $\sigma=(p, \sigma)$;
\item a $1$-morphism from $\sigma_1=(p_1, \sigma_1)$ to $\sigma_2=(p_2, \sigma_2)$ is a pair 
$$(\alpha, \int_{\sigma_1}^{\sigma_2} \G \alpha )$$ where
where $\alpha: p_1 \to p_2$ is a strict transformation and $\int_{\sigma_1}^{\sigma_2} \G $ is the obvious classifying object;
\item  a $2$-morphism  from $(\alpha, \int_{\sigma_1}^{\sigma_2} \G \alpha )$ to $(\alpha', \int_{\sigma_1}^{\sigma_2} \G \alpha' )$ is a pair
$$(\Gamma, \int_{\sigma_1}^{\sigma_2} \G \Gamma)  $$
where $\Gamma: \alpha \to \alpha'$ is a modification and 
$\int_{\sigma_1}^{\sigma_2} \G \Gamma: \int_{\sigma_1}^{\sigma_2} \G \alpha  \to \int_{\sigma_1}^{\sigma_2} \G \alpha' $ is the universal map; 
\item the composition functor  $$B(\F, \G)(\sigma_1, \sigma_2) \times B(\F,\G)(\sigma_2, \sigma_3) \to B(\F, \G)(\sigma_1, \sigma_3) $$ 
sends $$[(\alpha, \int_{\sigma_1}^{\sigma_2} \G \alpha ), (\beta, \int_{\sigma_2}^{\sigma_3} \G \beta)] \mapsto (\beta \otimes \alpha, \int_{\sigma_1}^{\sigma_3} \G (\beta \otimes \alpha))$$ 
and similarly 
$$[(\Gamma, \int_{\sigma_1}^{\sigma_2} \G \Gamma ), (\Gamma', \int_{\sigma_2}^{\sigma_3} \G \Gamma')] \mapsto (\Gamma'  \otimes \Gamma, \int_{\sigma_1}^{\sigma_3} \G (\Gamma' \otimes \Gamma)).$$ 
\end{itemize}

The $2$-category $B(\F,\G)$ will be called \textbf{the classifying $2$-category of morphisms from $\F$ to $\G$}.
\end{pdef}

\begin{proof}
Exercise.
\end{proof}
The choice of the letter `$B$' comes from the classifying space notation of homotopy theorists.
\paragraph{A forgetful functor} All the data in the definition of  $B(\F,\G)$ consist of `formal' pairs of some objects. The reader can check that the projection on the first factor of each of these data gives:
\begin{prop}\label{prop-canon-funct}
There is a canonical `forgetful'  strict $2$-functor 
$$ \pi: B(\F,\G) \to \dfunc(\C,\D).$$
This $2$-functor is locally a Grothendieck op-fibration.
\end{prop} 

\section{The canonical lax functor}
In what follows we want to establish that there is a canonical lax functor from the $2$-category $B(\F,\G)$ to $\M$. The key ingredient is the following lemma.
\begin{lem}\label{canon-lax-lem}
 If $\alpha: p_1 \to p_2$ and $\beta: p_2 \to p_3$ are two composable strict transformations, then the following holds.
\begin{enumerate}
\item there is a canonical map in $\M$:
$$ (\int_{\sigma_2}^{\sigma_3} \G \beta) \otimes ( \int_{\sigma_1}^{\sigma_2} \G \alpha)  \to \int_{\sigma_1}^{\sigma_3} \G (\beta \otimes \alpha).$$
\item these maps satisfy the `associativity' coherence condition i.e, the following commutes 
\[
\xy
(-60,10)*+{(\int_{\sigma_3}^{\sigma_4}\G \gamma)  \otimes (\int_{\sigma_2}^{\sigma_3} \G \beta) \otimes ( \int_{\sigma_1}^{\sigma_2} \G \alpha)}="A";
(0,18)*+{(\int_{\sigma_2}^{\sigma_4} \G \gamma \otimes \beta) \otimes ( \int_{\sigma_1}^{\sigma_2} \G \alpha)}="W";
(0,0)*+{(\int_{\sigma_3}^{\sigma_4} \G \gamma) \otimes ( \int_{\sigma_1}^{\sigma_3} \G \beta \otimes \alpha)}="X";
(50,10)*+{\int_{\sigma_1}^{\sigma_4}\G(\gamma \otimes \beta \otimes \alpha)} ="B";
{\ar@{->}^-{}"A";"W"};
{\ar@{->}^-{}"A";"X"};
{\ar@{->}^-{}"X";"B"};
{\ar@{->}^-{}"W";"B"};
\endxy
\]
\item If $\Gamma$ and $\Gamma'$ are composable modifications, then the following commutes
\[
\xy
(0,18)*+{ (\int_{\sigma_2}^{\sigma_3} \G \beta) \otimes ( \int_{\sigma_1}^{\sigma_2} \G \alpha)}="W";
(0,0)*+{(\int_{\sigma_2}^{\sigma_3} \G \beta) \otimes ( \int_{\sigma_1}^{\sigma_2} \G \alpha)}="X";
(50,18)*+{\int_{\sigma_1}^{\sigma_3} \G (\beta \otimes \alpha)}="B";
(50,0)*+{\int_{\sigma_1}^{\sigma_3} \G (\beta \otimes \alpha)}="C";
{\ar@{->}_-{}"W";"B"};
{\ar@{->}_-{}"X";"C"};
{\ar@{->}_-{(\int_{\sigma_2}^{\sigma_3} \G \Gamma') \otimes ( \int_{\sigma_1}^{\sigma_2} \G \Gamma)}"W";"X"};
{\ar@{->}^-{\int_{\sigma_2}^{\sigma_3} \G (\Gamma' \otimes \Gamma) }"B";"C"};
\endxy
\]
\end{enumerate}
\end{lem}
Before giving the proof of the lemma, we have a direct consequence:
\begin{cor}\label{cor-canon-lax}
The assignment:
\begin{itemize}
\item $\sigma \mapsto \ast$;
\item $(\alpha, \int_{\sigma_1}^{\sigma_2} \G \alpha) \mapsto \int_{\sigma_1}^{\sigma_2} \G \alpha $;

\item  $(\Gamma, \int_{\sigma_1}^{\sigma_2} \G \Gamma_c) \mapsto \int_{\sigma_1}^{\sigma_2} \G \Gamma_c $;
\end{itemize} 
defines a lax functor from  $B(\F,\G)$  to $\M$.
\end{cor}
\begin{proof}[Proof of Corollary \ref{cor-canon-lax}]
The previous lemma combined with Proposition \ref{prop-funct-integ} give the result.
\end{proof}
As the proof  of the lemma involves many diagrams we dedicate an entire paragraph for it.
\subsubsection{Proof of Lemma \ref{canon-lax-lem}}
Assertions $(2)$ and $(3)$ are consequences of Assertion $(1)$. Indeed,  Assertion $(2)$ follows from the coherence for the lax functor $\G$ and the uniqueness of factorization of maps through the equalizer. Assertion $(3)$ also follows from this fact that a compatible diagram passes in a unique way through the equalizer. It remains to prove Assertion $(1)$ which we discuss in the next paragraph.
\paragraph{Proof of  Assertion $(1)$}  The assertion will follow if we show that for each $f: c \to d$, the following parallel morphisms are the same:

$$ (\int_e \G \beta_e) \otimes ( \int_e \G \alpha_e) \xrightarrow{pr_c \otimes pr_c} \G \beta_c \otimes \G\alpha_c \to \G (\beta \otimes \alpha)_c \to \frac{\HOM_{\M}(\F(f), \G[\beta_d \otimes \alpha_d \otimes p_1(f)])}{ \HOM_{\M}(\F (f), \G[p_3(f) \otimes \beta_c \otimes \alpha_c])}$$

$$(\int_e \G \beta_e) \otimes ( \int_e \G \alpha_e) \xrightarrow{pr_d \otimes pr_d} \G \beta_d \otimes \G\alpha_d \to \G (\beta \otimes \alpha)_d \to \frac{\HOM_{\M}(\F(f), \G[\beta_d \otimes \alpha_d \otimes p_1(f)])}{ \HOM_{\M}(\F (f), \G[p_3(f) \otimes \beta_c \otimes \alpha_c])}. $$

Our desired map will be given by universal property of the equalizer. Now thanks to the adjunction $\otimes \dashv \HOM$ and the symmetry of $\otimes$, the previous maps are  equivalent to the ones below
$$ \int_e \G \beta_e \otimes  \int_e \G \alpha_e \otimes \F(f) \xrightarrow{(\Id \cdot pr_c \cdot pr_c) \circ \tx{sym}}  \F(f) \otimes \G \beta_c \otimes \G\alpha_c    \xrightarrow{\Id \otimes \tx{lax}} \F(f) \otimes \G (\beta \otimes \alpha)_c  \to \frac{\G[\beta_d \otimes \alpha_d \otimes p_1(f)]}{\G[p_3(f) \otimes \beta_c \otimes \alpha_c]}$$

$$ (\int_e \G \beta_e) \otimes ( \int_e \G \alpha_e) \otimes \F(f) \xrightarrow{pr_d \otimes pr_d \otimes \Id } \G \beta_d \otimes \G\alpha_d  \otimes \F(f)  \to \G (\beta \otimes \alpha)_d \otimes \F(f) \to \frac{\G[\beta_d \otimes \alpha_d \otimes p_1(f)]}{\G[p_3(f) \otimes \beta_c \otimes \alpha_c]}.$$

In order to prove that these new maps are equal, we use the hypothesis saying that for both $\alpha$ and $\beta$, we have commutative diagrams:
\[
\xy
(-30,10)*+{(\int_{\sigma_1}^{\sigma_2} \G \alpha_e)\otimes \F(f)}="A";
(0,18)*+{  \F(f) \otimes \G \alpha_c}="W";
(0,0)*+{\G \alpha_d \otimes \F(f)}="X";
(30,10)*+{\frac{\G[p_2(f) \otimes \alpha_c]}{\G[\alpha_d \otimes p_1(f)]}} ="B";
{\ar@{->}^-{}"A";"W"};
{\ar@{->}^-{}"A";"X"};
{\ar@{->}^-{}"X";"B"};
{\ar@{->}^-{}"W";"B"};
\endxy
\ \ \ \ \ 
\xy
(-30,10)*+{(\int_{\sigma_2}^{\sigma_3} \G \beta_e)\otimes \F(f)}="A";
(0,18)*+{\F(f) \otimes \G \beta_c}="W";
(0,0)*+{\G \beta_d \otimes \F(f)}="X";
(30,10)*+{\frac{\G[p_3(f) \otimes \beta_c]}{\G[\beta_d \otimes p_2(f)]}} ="B";
{\ar@{->}^-{}"A";"W"};
{\ar@{->}^-{}"A";"X"};
{\ar@{->}^-{}"X";"B"};
{\ar@{->}^-{}"W";"B"};
\endxy
\]
\ \\
\ \\
We give below the different steps that leads to the desired equality. 
\subparagraph*{Step 1} Tensor the first diamond by $\G \beta_d$ to get a new one:
\[
\xy
(-30,10)*+{\G \beta_d \otimes (\int_{\sigma_1}^{\sigma_2} \G \alpha_e)\otimes \F(f)}="A";
(0,18)*+{  \G \beta_d \otimes \F(f) \otimes \G \alpha_c}="W";
(0,0)*+{\G \beta_d \otimes \G \alpha_d \otimes \F(f)}="X";
(30,10)*+{\G \beta_d \otimes \frac{\G[p_2(f) \otimes \alpha_c]}{\G[\alpha_d \otimes p_1(f)]}} ="B";
(80,10)*+{ \frac{\G (\beta_d \otimes p_2(f) \otimes \alpha_c)}{\G (\beta_d \otimes \alpha_d \otimes p_1(f))}}="C";
{\ar@{->}^-{}"A";"W"};
{\ar@{->}^-{}"A";"X"};
{\ar@{->}^-{}"X";"B"};
{\ar@{->}^-{}"W";"B"};
{\ar@{-->}^-{}"B";"C"};
%%%%%%%%% lower extension%%
(30,-10)*+{\G [\beta_d \otimes \alpha_d] \otimes \G p_1(f)} ="E";
{\ar@{.>}^-{}"X";"E"};
{\ar@{.>}^-{}"E";"C"};
\endxy
\]

\subparagraph*{Step 2} Tensor the second diamond by $\G \alpha_c$ to get this one:
\[
\xy
(-30,10)*+{(\int_{\sigma_2}^{\sigma_3} \G \beta_e)\otimes \F(f) \otimes \G \alpha_c}="A";
(0,18)*+{\F(f) \otimes \G \beta_c \otimes \G \alpha_c}="W";
(0,0)*+{\G \beta_d \otimes \F(f) \otimes \G \alpha_c}="X";
(30,10)*+{\frac{\G[p_3(f) \otimes \beta_c]}{\G[\beta_d \otimes p_2(f)]} \otimes \G \alpha_c} ="B";
(80,10)*+{ \frac{\G (p_3(f) \otimes \beta_c \otimes \alpha_c )}{\G (\beta_d \otimes p_2(f) \otimes \alpha_c)}}="C";
{\ar@{->}^-{}"A";"W"};
{\ar@{->}^-{}"A";"X"};
{\ar@{->}^-{}"X";"B"};
{\ar@{->}^-{}"W";"B"};
{\ar@{-->}^-{}"B";"C"};
%%%%%%%%% lower extension%%
(30,-10)*+{\G \beta_d \otimes \G[ p_2(f) \otimes  \alpha_c]} ="E";
{\ar@{.>}^-{}"X";"E"};
{\ar@{.>}^-{}"E";"C"};
\endxy
\]

\subparagraph*{Step 3} The bifunctoriality and the symmetry of $\otimes$ allows us to write the `mixed projection':
$$ (\int_{\sigma_2}^{\sigma_3} \G \beta_e) \otimes ( \int_{\sigma_1}^{\sigma_2} \G \alpha_e) \otimes \F(f) \xrightarrow{(pr_d \otimes \Id \otimes pr_c) \circ (\Id \otimes \tx{sym}) } \G \beta_d \otimes \F(f) \otimes \G\alpha_c $$

as a composite:

\[
\xy
(-50,10)*+{(\int_{\sigma_2}^{\sigma_3} \G \beta_e) \otimes ( \int_{\sigma_1}^{\sigma_2} \G \alpha_e) \otimes \F(f)}="A";
(0,18)*+{(\int_{\sigma_2}^{\sigma_3} \G \beta_e) \otimes \G\alpha_c  \otimes \F(f)  }="W";
(0,0)*+{\G \beta_d \otimes ( \int_{\sigma_1}^{\sigma_2} \G \alpha_e) \otimes \F(f) }="X";
(30,10)*+{\G \beta_d \otimes \F(f) \otimes \G\alpha_c } ="B";
{\ar@{->}^-{}"A";"W"};
{\ar@{->}^-{}"A";"X"};
{\ar@{->}^-{}"X";"B"};
{\ar@{->}^-{}"W";"B"};
\endxy
\]

\subparagraph*{Step 4} In the diagram below, the `bounding paths' are exactly the morphisms we want to compare; and since this diagram has diamonds that commute, we get the desired equality.

\[
\xy
(-30,10)+(-5,0)*+{(\int_{\sigma_2}^{\sigma_3} \G \beta_e) \cdot ( \int_{\sigma_1}^{\sigma_2} \G \alpha_e) \cdot \F(f)}="A";
(0,18)*+{(\int_{\sigma_2}^{\sigma_3} \G \beta_e) \cdot \G \alpha_c \cdot \F(f)}="W";
(0,0)*+{\G \beta_d \cdot ( \int_{\sigma_1}^{\sigma_2} \G \alpha_e) \cdot \F(f)}="X";
(30,10)+(5,0)*+{\G \beta_d  \cdot  \F(f) \cdot \G \alpha_c } ="B";
(88,10)*+{ \frac{\G (p_3(f) \cdot \beta_c \cdot \alpha_c )}{\G (\beta_d \cdot p_2(f) \cdot \alpha_c)}=\frac{\G (\beta_d \cdot p_2(f) \cdot \alpha_c)}{\G (\beta_d \cdot \alpha_d \cdot p_1(f))}}="C";
{\ar@{->}^-{}"A";"W"};
{\ar@{->}^-{}"A";"X"};
{\ar@{->}^-{}"X";"B"};
{\ar@{->}^-{}"W";"B"};
{\ar@{-->}^-{}"B";"C"};
%%%%%%%%% lower extension%%
(30,-10)*+{\G \beta_d \cdot \G \alpha_d \cdot  \F(f)} ="E";
{\ar@{.>}^-{}"X";"E"};
{\ar@{.>}^-{}"E";"C"};
%%%%%%%%% upper extension%%
(30,28)*+{\F(f) \cdot \G \beta_c \cdot \G \alpha_c}="F";
{\ar@{.>}^-{}"W";"F"};
{\ar@{.>}^-{}"F";"C"};
%%%%%%%%% labelling diamonds%%
(30,28)+(10,-10)*+{\tx{Step $2$}};
(30,-10)+(10,10)*+{\tx{Step $1$}};
(0,0)+(0,9)*+{\tx{Step $3$}};
\endxy
\]

The equality of the two maps induces, by universal property of the equalizer, a \textbf{unique} map 
$$ (\int_{\sigma_2}^{\sigma_3} \G \beta_e) \otimes ( \int_{\sigma_1}^{\sigma_2} \G\alpha_e)  \to\int_{\sigma_1}^{\sigma_3} \G (\beta \otimes \alpha)_e$$
which makes everything compatible. This ends the proof of Assertion $(1)$. $\qed$
\begin{df}
Let $\F$, $\G$ and $\M$ be as previously.\ \\

Define the \textbf{internal $\HOM$} of $\F$ and $\G$, as the previous lax functor 
$$\HOM(\F,\G): B(\F,\G) \to \M.$$
\end{df} 

\subsection{The canonical evaluation} 
Our goal here is to establish the following proposition.

\begin{prop}\label{eval_bc}
We have a canonical \ul{evaluation} morphism
$$ \F \boxtimes \HOM(\F,\G) \to \G.$$
\end{prop}

\subsubsection{Proof of Proposition \ref{eval_bc}}
Recall that in order to define $\F \boxtimes B(\C,\D) \to \G$, we need two ingredients: a $2$-functor 
$$\gamma: \C \times B(\F,\G)  \to \D$$ and a transformation 
$$\sigma: \F \boxtimes \Hom(\F,\G) \to \gamma^{\star} \G$$ 
in $\Lax(\C \times B(\F,\G), \M)$. \ \\

Without a surprise the $2$-functor $\gamma$ is the the composite:
$$\gamma:\C \times B(\F,\G) \xrightarrow{\Id \times \pi} \C \times \dfunc(\C,\D) \xrightarrow{\tx{ev}} \D$$
where $\pi$ is the forgetful $2$-functor of Proposition \ref{prop-canon-funct}

In particular if $\alpha: p_1 \to p_2$ is a \ul{strict} transformation, and $f: c \to d$ is a $1$-morphism in $\C$, then $\gamma(f,\alpha)$ is:
$$ \alpha_d \otimes p_1(f) = p_2(f) \otimes \alpha_c \in \D(p_1(c), p_2(d)).$$
\ \\

In order to define the transformation $\sigma:\F \boxtimes B(\C,\D) \to \gamma^{\star} \G$, we need two ingredients:
\begin{enumerate}
\item specify a map $$\sigma_{\alpha,f}:\F(f) \otimes B(\C,\D)(\alpha) \to \G \gamma(f, \alpha)$$ that is a map
$$\F(f) \otimes (\int_a \G \alpha_a)  \to  \frac{\G[\alpha_d \otimes p_1(f)] }{\G[p_2(f) \otimes \alpha_c]}$$
\item  check that given $f: c \to d$, $h: d \to e$, $\alpha: p_1 \to p_2$ and $\beta: p_2 \to p_3$,  the following commutes:
\begin{align}\label{eq1}
\xy
(-40,18)*+{[\F(h) \cdot \int_a \G \beta_a] \cdot  [\F(f) \cdot \int_a \G \alpha_a]}="W";
(-40,0)*+{ \frac{\G[\beta_e \cdot p_2(h)] }{\G[p_3(h) \cdot \beta_d]}\cdot \frac{\G[\alpha_d \cdot p_1(f)] }{\G[p_2(f) \cdot \alpha_c]}  }="X";
(0,38)*+{[\F(h) \cdot \F(f)] \cdot  [\int_a \G \beta_a \cdot \int_a \G \alpha_a]}="E";
(40,18)*+{\F(h \cdot f) \cdot \int_e \G (\beta \cdot \alpha)_e}="B";
(40,0)*+{\frac{\G[\beta_e \cdot p_2(h) \cdot \alpha_d \cdot p_1(f)] }{\G[p_3(h) \cdot \beta_d \cdot p_2(f) \cdot \alpha_c]}=\frac{\G[\beta_e \cdot \alpha_e \cdot p_1(h\cdot f)] }{\G[p_3(h \cdot f) \cdot \beta_c \cdot \alpha_c]}}="C";
{\ar@{->}_-{\sigma_{\beta,h} \cdot \sigma_{\alpha,f}}"W";"X"};
{\ar@{.>}^-{}"W";"E"};
{\ar@{.>}^-{}"E";"B"};
{\ar@{->}^-{}"W";"B"};
{\ar@{->}^-{}"X";"C"};
{\ar@{->}^-{\sigma_{\beta \cdot \alpha,h\cdot f}}"B";"C"};
\endxy
\end{align}
\end{enumerate}

The map $$\sigma_{\alpha,f}:\F(f) \otimes (\int_a \G \alpha_a)  \to  \frac{\G[\alpha_d \otimes p_1(f)] }{\G[p_2(f) \otimes \alpha_c]}$$
 is the canonical map mentioned in Remark \ref{rmk_eval}: 
$$ \F(f) \otimes (\int_a \G \alpha_a)  \to  \frac{\G[\alpha_d \otimes p_1(f)] }{\G[p_2(f) \otimes \alpha_c]}.$$

In order to have the above commutative diagram we proceed as follows.

\paragraph{Step 1} We have a canonical evaluation map displayed as a commutative diagram:
\begin{align*}
\xy
(-40,10)*+{(\int_e \G (\beta \otimes \alpha)_e)\otimes \F(h\otimes f)}="A";
(-10,20)*+{ \F(h\otimes f) \otimes \G (\beta \otimes \alpha)_c}="W";
(10,20)+(30,0)*+{ \G [p_3(h\otimes f)] \otimes \G (\beta \otimes \alpha)_c}="C";
(-10,0)*+{\G (\beta \otimes \alpha)_e \otimes \F(h\otimes f))}="X";
(10,0)+(30,0)*+{\G (\beta \otimes \alpha)_e \otimes \G [p_1(h\otimes f)]}="Y";
(40,10)+(30,0)*+{\frac{\G[\beta_e \otimes \alpha_e \otimes p_1(h\otimes f)] }{\G[p_3(h \otimes f) \otimes \beta_c \otimes \alpha_c]}} ="B";
{\ar@{->}^-{}"A";"W"};
{\ar@{->}^-{}"A";"X"};
{\ar@{->}^-{}"W";"C"};
{\ar@{->}^-{}"C";"B"};
{\ar@{->}^-{}"X";"Y"};
{\ar@{->}^-{}"Y";"B"};
\endxy
\end{align*}

\paragraph{Step 2} By construction of the laxity map 
$$ (\int_e \G \beta_e) \otimes ( \int_e \G \alpha_e)  \to \int_e \G (\beta \otimes \alpha)_e$$
the following commutes:
\begin{align*}
\xy
(-60,0)+(-10,0)*+{[\int_a \G \beta_a \cdot \int_a \G \alpha_a ] \cdot \F(h \cdot f) }="X";
(-5,0)+(-10,0)*+{[\G \beta_e \cdot \G \alpha_e]\cdot \F(h \cdot f)}="C";
{\ar@{->}^-{}"X";"C"};
%%% lower square%%%
(-60,0)+(-20,-15)*+{\int_a \G (\beta \cdot \alpha)_a  \cdot \F(h \cdot f) }="P";
(-5,0)+(-20,-15)*+{\G (\beta_e \cdot \alpha_e)\cdot \F(h \cdot f)}="Q";
(50,0)+(0,0)+(-20,-15)*+{\G [\beta_e \cdot \alpha_e \cdot p_1(h \cdot f)]}="R";
{\ar@{->}^-{}"X";"P"};
{\ar@{->}^-{}"P";"Q"};
{\ar@{->}^-{}"C";"Q"};
{\ar@{->}^-{}"Q";"R"};
\endxy
\end{align*}

\paragraph{Step 3} Both $\sigma_{\beta, h}$ and $\sigma_{\alpha, f}$ are composite of two parallel maps and if we use the two presentations of each, one can establish that we have a commutative diagram:

\begin{align*}
\xy
(-40,10)+(-60,0)*+{[\int_a \G \beta_a \cdot \F(h)] \cdot [\int_a \G \alpha_a \cdot \F(f)]}="A";
(-10,20)+(-60,0)+(0,10)+(-10,0)*+{[\G \beta_e\cdot  \F(h)]\cdot  [\int_a \G \alpha_a \cdot \F(f)]}="W";
(-10,0)+(-60,0)*+{[\F(h) \cdot \G \beta_d] \cdot [\int_a \G \alpha_a \cdot \F(f)]}="X";
(30,10)+(-60,0)*+{\frac{\G[\beta_e \cdot p_2(h)] }{\G[p_3(h) \cdot \beta_d]} \cdot [\int_a \G \alpha_a \cdot \F(f)]} ="B";
{\ar@{->}^-{}"A";"W"};
{\ar@{->}^-{}"A";"X"};
{\ar@{->}^-{}"W";"B"};
{\ar@{->}^-{}"X";"B"};
%%%%%%second diamond %%%%%%%
(40,10)+(-60,0)+(30,10)*+{\G[\beta_e \cdot p_2(h)] \cdot  [\G \alpha_d \cdot \F(f)]} ="C";
(40,10)+(-60,0)+(30,-10)*+{\G[\beta_e \cdot p_2(h)] \cdot  [\F(f) \cdot \G \alpha_c]} ="D";
(30,10)+(-60,0)+(70,0)*+{\G[\beta_e \cdot p_2(h)] \cdot \frac{\G [\alpha_d \cdot p_1(f)] }{\G[ p_2(f) \cdot \alpha_c]}} ="E";
{\ar@{->}^-{}"B";"C"};
{\ar@{->}^-{}"B";"D"};
{\ar@{->}^-{}"D";"E"};
{\ar@{->}^-{}"C";"E"};
%%%%%%%%%%% diamond above WMNO%%%%%
(-10,20)+(-60,0)+(30,30)*+{\G \beta_e\cdot [\G \alpha_e \cdot \F(h)] \cdot \F(f)}="M";
(-10,20)+(-60,0)+(50,10)+(-10,0)*+{\G \beta_e\cdot [\F(h)\cdot  \G \alpha_d] \cdot \F(f)}="N";
(-10,20)+(-60,0)+(50,10)+(20,10)+(-5,0)*+{\G \beta_e\cdot [\frac{\G[\alpha_e \cdot p_1(h)]}{\G[p_2(h) \cdot \alpha_d]}] \cdot \F(f)}="O";
{\ar@{->}^-{}"W";"M"};
{\ar@{->}^-{}"W";"N"};
{\ar@{->}^-{}"N";"O"};
{\ar@{->}^-{}"M";"O"};
%%% final point %%%%
(-10,20)+(-60,0)+(50,10)+(20,10)+(40,0)*+{\frac{\G[ \beta_e \cdot \alpha_e \cdot p_1(h) \cdot p_1(f)]}{\G[\beta_e \cdot p_2(h) \cdot \alpha_d \cdot p_1(f)]}}="Z";
{\ar@{->}^-{}"O";"Z"};
{\ar@{->}^-{}"C";"Z"};
{\ar@{->}^-{}"E";"Z"};
\endxy
\end{align*}
\paragraph{Step 4} Build the following commutative diagram
\begin{align*}
\xy
(-40,18)+(-40,0)*+{[\int_a \G \beta_a \cdot \F(h) ] \cdot  [\int_a \G \alpha_a  \cdot \F(f) ]}="A";
(-40,18)+(-40,0)+(60,0)*+{[\int_a \G \beta_a \cdot \int_a \G \alpha_a ] \cdot  [\F(h) \cdot \F(f) ]}="W";
(-60,0)+(-10,0)+(-10,0)*+{[\int_a \G \beta_a \cdot \int_a \G \alpha_a ] \cdot \F(h \cdot f) }="X";
(30,18)+(-40,0)+(50,0)*+{[\G \beta_e \cdot \G \alpha_e]  \cdot  [\G p_1(h) \cdot \G p_1(f) ] }="B";
(-5,0)+(-10,0)+(-10,0)+(-5,0)*+{[\G \beta_e \cdot \G \alpha_e]\cdot \F(h \cdot f)}="C";
{\ar@{->}_-{}"W";"X"};
{\ar@{->}^-{\cong}"A";"W"};
{\ar@{->}^-{}"W";"B"};
{\ar@{->}^-{}"X";"C"};
%%%% modif%%%
(50,0)+(-10,0)+(-25,0)*+{[\G \beta_e \cdot \G \alpha_e]\cdot \G [p_1(h \cdot f)]}="D";
(50,0)+(-10,0)+(20,0)*+{\G \beta_e \cdot \G[\alpha_e  \cdot  p_1(h)] \cdot \G p_1(f)}="E";
{\ar@{->}^-{}"B";"D"};
{\ar@{->}^-{}"B";"E"};
{\ar@{->}^-{}"C";"D"};
%%% lower square%%%
(-60,0)+(-20,-15)*+{\int_a \G (\beta \cdot \alpha)_a  \cdot \F(h \cdot f) }="P";
(-5,0)+(-20,-15)*+{\G (\beta_e \cdot \alpha_e)\cdot \F(h \cdot f)}="Q";
(50,0)+(0,0)+(-20,-15)*+{\G [\beta_e \cdot \alpha_e \cdot p_1(h \cdot f)]}="R";
{\ar@{->}^-{}"X";"P"};
{\ar@{->}^-{}"P";"Q"};
{\ar@{->}^-{}"C";"Q"};
{\ar@{->}^-{}"Q";"R"};
{\ar@{->}^-{}"D";"R"};
{\ar@{->}^-{}"E";"R"};
\endxy
\end{align*}
\ \\
\ \\
The equality of the bouding paths gives us the commutativity of  (\ref{eq1}) as desired. $\qed$

\section{The isomorphism $\Hom[\F \boxtimes \Ea, \G] \cong \Hom[\Ea, \HOM(\F,\G)]$}
In the following we consider three lax functors $\Ea: \Ba \to \M$, $\F: \C \to \M$ and $\G: \D \to \M$. \ \\

Let $\gamma: \F \boxtimes \Ea \to \G$ be a morphism in $\M_{lax}(\dCat)$. We want to show that there exists a \textbf{unique} morphism $[\gamma]: \Ea \to \HOM(\F,\G)$, such that $\gamma$ is the composite:
\begin{align}\label{eq2}
\F \boxtimes \Ea \xrightarrow{\Id \boxtimes [\gamma]} \F \boxtimes \HOM(\F,\G) \xrightarrow{ev} \G.
 \end{align}

By definition $\gamma$ is given by a $2$-functor $p: \C \times \Ba \to \D$ together with a transformation (an icon) $\gamma: \F \boxtimes \Ea \to p^{\star} \G$. 
\subsection{The underlying functor of $[\gamma]$}
We must show that there is a canonical (unique) $2$-functor between the respective indexing $2$-categories of $\Ea$ and $\HOM(\F,\G)$.\\

Let us recall that  in $\dCat$ with strict $2$-functors, \ul{strict} transformations and modification, we have an internal $\Hom$ which is precisely $\dfunc(\C,\D)$.
So $p : \C \times \Ba \to \D$ has a unique adjoint transpose $[p]: \Ba \to \dfunc(\C,\D)$ such that $p$ is the composite:
$$ \C \times \Ba \xrightarrow{\Id \times [p] } \C \times \dfunc(\C,\D) \xrightarrow{ev} \D.$$ 

If $b$ is an object of $\Ba$,  $[p]_b$ is the partial $2$-functor $p(-,b)$, while for a $1$-morphism $k: b \to b'$, $[p]_k: p(-,b) \to p(-,b')$ is the obvious \ul{strict} transformation. 
\paragraph{Partial transformations}
It's easy to see that we also have \emph{partial transformations} 
$$\gamma(-,b): \F \to \G$$ for every $b \in \Ba$. The components of $\gamma(-,b)$ are given by the maps:
$$ \F(f) \cong \F(f) \otimes I \xrightarrow{\Id \otimes \varphi_b} \F(f) \otimes \Ea(\Id_b) \xrightarrow{\gamma(f,\Id_b)} \G[p(f,\Id_b)]= \G([p]_b(f)).$$
Here $\varphi_b: I \to \Ea(\Id_b)$ is the laxity map of the identity.

\begin{nota}
In order to keep a coherence with the previous notations, we will write:
\begin{itemize}
\item $\sigma_b= \gamma(-,b)$;
\item  $\alpha_{k}=[p]_k$ (for $k : b \to b'$);
\item $\Gamma_{\theta}=$ the modification corresponding to a $2$-morphism $\theta: k \to k'$.
\end{itemize}
\end{nota}

Recall that the indexing $2$-category of $\HOM(\F,\G)$ is $B(\F,\G)$ and thanks to Proposition \ref{prop-canon-funct} there is a canonical $2$-functor
$$ \pi: B(\F,\G) \to \dfunc(\C,\D) $$
which is locally a Grothendieck op-fibration.

\begin{lem}
There exists a unique strict $2$-functor, `a lift' 
$$\delta: \Ba \to B(\F,\G) $$
such that the following commutes:
\[
\xy
(-10,0)*+{\Ba}="X";
(20,0)*+{\dfunc(\C,\D)}="Y";
(20,18)*+{B(\F,\G)}="E";
{\ar@{->}^-{[p]}"X";"Y"};
{\ar@{->}^-{\pi }"E";"Y"};
{\ar@{.>}^-{\delta }"X";"E"};
\endxy
\]
\end{lem}

\begin{proof}[Sketch of proof]
The $2$-functor $\delta$ sends:
\begin{itemize}[label=$-$]
\item $b$  to $([p]_b, \sigma_b)$;
\item $k$ to  $(\alpha_{k}, \int_{\sigma_b}^{\sigma_{b'}} \G \alpha_{k}) $
\item $\theta$ to $(\Gamma_{\theta}, \int_{\sigma_b}^{\sigma_{b'}} \G \Gamma_{\theta})$.
\end{itemize}
The verifications that these data define a $2$-functor is tedious but straightforward. The fact that $\delta$ is a lift is also straightforward.
\end{proof}

Our first step toward the adjunction is to establish that:
\begin{lem}\label{lem_adj}
With the above notation, given a morphism $\gamma: \F \boxtimes \Ea \to \G$, the following holds.
\begin{enumerate}
\item For every $1$-morphism $k$ of $\Ba$ we have a canonical map in $\M$
$$ \Ea(k) \to \int_a \G \alpha_{k,a}.$$ 
\item For every $2$-morphism $\theta : k \to k'$ we have a commutative square
\[
\xy
(-10,18)*+{\Ea(k)}="W";
(-10,0)*+{\Ea(k')}="X";
(20,0)*+{\int_a \G \alpha_{k'}}="Y";
(20,18)*+{\int_a \G\alpha_k }="E";
{\ar@{->}^-{}"X";"Y"};
{\ar@{->}_-{\Ea(\theta)}"W";"X"};
{\ar@{->}^-{\int_a \G \Gamma_{\theta} }"E";"Y"};
{\ar@{->}^-{}"W";"E"};
\endxy
\]
\end{enumerate} 
\end{lem}

\subsection{Proof of Lemma \ref{lem_adj}}
Assertion $(2)$ will follow by standard arguments: it's a functoriality of universal constructions, so we leave it as an exercise for the reader. We only sketch the proof of Assertion $(1)$. \ \\

Using $\gamma: \F \boxtimes \Ea \to p^{\star} \G$  for the $1$-morphism $(\Id_c, k) \in \C \times \Ba$ we get a map in $\M$:
$$\gamma(\Id_c, k): \F(\Id_c) \otimes \Ea(k) \to \G[p(\Id_c, k)].$$

Now  on the one hand $p(\Id_c, k)$ is precisely the component of $\alpha_k = [p]_k$ at $c$, thus we can write 
 $$\G[p(\Id_c, k)]= \G \alpha_{k,c}.$$  
 On the other hand we have the laxity map $ \varphi_c: I \to \F(\Id_c)$ and we can use it to  get a map
 $$\epsilon_c: \Ea(k) \xrightarrow{\cong} I \otimes \Ea(k) \xrightarrow{\varphi_c \otimes \Id} \F(\Id_c) \otimes \Ea(k) \xrightarrow{\gamma(\Id_c, k)} \G \alpha_{k,c}.$$
\begin{sublem}
For every $1$-morphism $f:c \to d$ of $\C$ the following diagram commutes.
\begin{align*}
\xy
(-30,10)*+{\F(f) \otimes \Ea(k)}="A";
(0,18)*+{  \F(f) \otimes \G \alpha_{k,c}}="W";
(0,0)*+{\G \alpha_{k,d} \otimes \F(f)}="X";
(30,10)+(10,0)*+{\frac{\G[[p]_{b'}(f) \otimes \alpha_{k,c}]}{\G[\alpha_{k,d} \otimes [p]_{b}(f)]}=\G[p(f, k)]} ="B";
{\ar@{->}^-{}"A";"W"};
{\ar@{->}^-{}"A";"X"};
{\ar@{->}^-{}"X";"B"};
{\ar@{->}^-{}"W";"B"};
\endxy
\end{align*}

Moreover the two composites in this diagram are exactly the map:
$$  \gamma(f, k): \F(f) \otimes \Ea(k) \to \G[p(f,k)].$$
\end{sublem} 
Before giving the proof of the sub-lemma, here is how we get a proof of Assertion $(1)$ and hence of the lemma.

\begin{proof}[Proof of Assertion $(1)$]
From the sub-lemma, we have by universal property of $\int_a \G \alpha_{k,a}$, which is an equalizer, a \textbf{unique map} $ [\gamma]_k: \Ea(k) \to  \int_a \G \alpha_{k,a}$ such that we can factorize $\gamma(f,k)$ as:
$$\F(f) \otimes \Ea(k) \xrightarrow{\Id \otimes [\gamma]_k} \F(f) \otimes  \int_a \G \alpha_{k,a} \xrightarrow{\tx{ev}} \G[p(f,k)].$$
\end{proof}
\begin{rmk}\label{rmk-fact}
The above factorization is the one that will imply the  factorization (\ref{eq2}).
\end{rmk}

\begin{proof}[Proof of the Sub-lemma]
The $1$-morphism $(f,k)$ can be written in two ways as a composite:
\begin{align*}
(f,k)= (f, \Id_b') \otimes (\Id_c, k) \\
(f,k)= (\Id_d, k) \otimes (f, \Id_b) 
\end{align*}
And so we have two commutative squares, obtained from the coherence of the components of $\gamma: \F \boxtimes \Ea \to \G$:
\begin{align}\label{diag_sublem_1}
\xy
(-10,18)*+{[\F(f) \otimes \Ea(\Id_{b'})] \otimes [\F(\Id_c) \otimes \Ea(k)]}="W";
(-10,0)*+{\G[p(f,\Id_{b'})] \otimes \G[p(\Id_c, k)]}="X";
(35,0)*+{\G[p(f,k)]}="Y";
(35,18)*+{ \F(f) \otimes \Ea(k)}="E";
{\ar@{->}^-{}"X";"Y"};
{\ar@{->}_-{}"W";"X"};
{\ar@{->}^-{\gamma(f,k)}"E";"Y"};
{\ar@{->}^-{}"W";"E"};
\endxy
\end{align}
\begin{align}\label{diag_sublem_2}
\xy
(-10,18)*+{[\F(\Id_d) \otimes \Ea(k)] \otimes [\F(f) \otimes \Ea(\Id_{b})]}="W";
(-10,0)*+{\G[p(\Id_d, k)] \otimes \G[p(f,\Id_{b})]}="X";
(35,0)*+{\G[p(f,k)]}="Y";
(35,18)*+{ \F(f) \otimes \Ea(k)}="E";
{\ar@{->}^-{}"X";"Y"};
{\ar@{->}_-{}"W";"X"};
{\ar@{->}^-{\gamma(f,k)}"E";"Y"};
{\ar@{->}^-{}"W";"E"};
\endxy
\end{align}

Both lax functors $\F$ and $\Ea$ are subjected to coherences with respect to the laxity map of identity $1$-morphisms. For example for $\F$ we have commutative diagrams (for the left unity axiom):

\begin{align*}
\xy
(-10,18)*+{I \otimes \F(f)}="W";
(-10,0)*+{\F(\Id_d) \otimes \F(f)}="X";
(20,0)*+{\F(f)}="Y";
(20,18)*+{\F(f)}="E";
{\ar@{->}^-{\tx{laxity}}"X";"Y"};
{\ar@{->}_-{\varphi_d \otimes \Id}"W";"X"};
{\ar@{->}^-{\F(\Id_f)}_{=}"E";"Y"};
{\ar@{->}^-{\cong}"W";"E"};
\endxy
\end{align*}
\ \\

If we tensor edge by edge the corresponding diagrams for $\F(f)$ and $\Ea(k)$ we get two commutative diagrams:
\begin{align}\label{diag_sublem_3}
\xy
(-10,18)*+{[\F(f) \otimes I] \otimes [I \otimes \Ea(k)]}="W";
(-10,0)*+{[\F(f) \otimes \F(\Id_c)  \otimes [\Ea(\Id_{b'}) \otimes \Ea(k)]}="X";
(35,0)*+{\F(f) \otimes \Ea(k)}="Y";
(35,18)*+{ \F(f) \otimes \Ea(k)}="E";
{\ar@{->}^-{}"X";"Y"};
{\ar@{->}_-{}"W";"X"};
{\ar@{=}^-{}"E";"Y"};
{\ar@{->}^-{\cong}"W";"E"};
%%%%
(-10,0)+(-65,0)*+{[\F(f) \otimes \Ea(\Id_{b'})] \otimes [\F(\Id_c) \otimes \Ea(k)]}="Z";
{\ar@{->}^-{\cong}"Z";"X"};
{\ar@{.>}^-{}"W";"Z"};
\endxy
\end{align}
\begin{align}\label{diag_sublem_4}
\xy
(-10,18)*+{[I \otimes  \F(f)] \otimes [\Ea(k) \otimes I] }="W";
(-10,0)*+{[ \F(\Id_d) \otimes \F(f)] \otimes [\Ea(k) \otimes \Ea(\Id_b)] }="X";
(35,0)*+{ \F(f) \otimes \Ea(k)}="Y";
(35,18)*+{ \F(f) \otimes \Ea(k)}="E";
{\ar@{->}^-{}"X";"Y"};
{\ar@{->}_-{}"W";"X"};
{\ar@{=}^-{}"E";"Y"};
{\ar@{->}^-{\cong}"W";"E"};
%%%%
(-10,0)+(-65,0)*+{[\F(\Id_d) \otimes \Ea(k)] \otimes [\F(f) \otimes \Ea(\Id_{b})]}="Z";
{\ar@{->}^-{\cong}"Z";"X"};
{\ar@{.>}^-{}"W";"Z"};
\endxy
\end{align} 
\ \\

The dotted arrows above are respectively $(\Id \otimes \varphi_{b'}) \otimes (\varphi_c \otimes \Id)$ and  $ (\varphi_d \otimes \Id) \otimes ( \Id \otimes \varphi_b)$. The bottom horizontal arrows in (\ref{diag_sublem_3}) and (\ref{diag_sublem_4}) are respectively the top horizontal arrows in (\ref{diag_sublem_1}) and (\ref{diag_sublem_2}). So we can glue  (\ref{diag_sublem_1}) and (\ref{diag_sublem_3}); and glue (\ref{diag_sublem_2}) and (\ref{diag_sublem_4}) along the respective horizontal arrow. This  gives new commutative diagrams and we leave the reader to check that the commutativity of these new diagrams implies that the diamond of the sub-lemma commutes and that the composite is indeed $\gamma(f,k)$.
\end{proof}
The following lemma plays an important role since it says that:
\begin{lem}
With above notation, the maps $$ \Ea(k) \to \int_a \G \alpha_{k,a}$$
together with the lift  $\delta: \Ba \to B(\F,\G) $:
\[
\xy
(-10,0)*+{\Ba}="X";
(20,0)*+{\dfunc(\C,\D)}="Y";
(20,18)*+{B(\F,\G)}="E";
{\ar@{->}^-{[p]}"X";"Y"};
{\ar@{->}^-{\pi }"E";"Y"};
{\ar@{.>}^-{\delta }"X";"E"};
\endxy
\]
define a morphism $[\gamma]: \Ea \to \Hom(\F,\G)$
\end{lem}

\begin{proof}[Sketch of proof]
With all the previous material, we simply have to check that we have commutative diagrams: 
\begin{align*}
\xy
(-20,18)*+{\Ea(k) \otimes \Ea(k')}="W";
(-20,0)*+{\int_a \G \alpha_{k,a} \otimes \int_a \G \alpha_{k',a}}="X";
(20,0)*+{\int_a \G \alpha_{k\otimes k',a}}="Y";
(20,18)*+{\Ea(k\otimes k')}="E";
{\ar@{->}^-{}"X";"Y"};
{\ar@{->}_-{}"W";"X"};
{\ar@{->}^-{}_{}"E";"Y"};
{\ar@{->}^-{}"W";"E"};
\endxy
\end{align*}

This is not hard but tedious to check. One needs to use the following facts:
\begin{itemize}[label=$-$]
\item $\gamma$ is a morphism of lax functors; and therefore we have similar commutative diagrams; and 
\item the fact that for any $f: c \to d$ we have two presentations of the pair $(f, k\otimes k')$:
\begin{align*}
(f,k \otimes k')= (f, k) \otimes (\Id, k') \\
(f,k \otimes k')= (\Id, k) \otimes (f, k') 
\end{align*}
\end{itemize}

By the same techniques as in the proof of the previous sub-lemma, we build a (big) commutative diagram (starting at $ \F(f) \otimes \Ea(k) \otimes \Ea(k')$) . The universal property of $\int_a \G \alpha_{k\otimes k',a}$ gives the result (the uniqueness of maps going into an equalizer). We leave the details to the reader. 
\end{proof}
%%%%%%%%%%%%%%%%%%%%%%%%%%%%%%%%%%%%%%%%%%%%%%% 
%%%%%%%%%%%%%%%%%%%%%%%%%%%%%%%%%%%%%%%%%%%%%%% 
%%%%%%%%%%%%%%%%%%%%%%%%%%%%%%%%%%%%%%%%%%%%%%% 
\subsection{The main theorem}
\begin{thm}\label{thm-adjunction}
Let  $\Ea: \Ba \to \M$, $\F: \C \to \M$ and $\G: \D \to \M$
be three lax functors in a complete and cocomplete symmetric monoidal category $\M$.\\

Then there is an isomorphism of category 
$$\Hom[\F \boxtimes \Ea, \G] \cong \Hom[\Ea, \HOM(\F,\G)]$$
which is functorial in the $3$ variables. 
\end{thm}

\begin{proof}
Combine the previous lemma together with Remark \ref{rmk-fact}.
\end{proof}

\bibliographystyle{plain}
\bibliography{Bibliography_These}
\end{document}